\documentclass{amsart}
\usepackage{amssymb,amsfonts}
\usepackage[all,arc]{xy}
\usepackage{enumerate}
\usepackage{mathrsfs,multicol}
\usepackage[margin=1in]{geometry}

\newtheorem{thm}{Theorem}[section]
\newtheorem{cor}[thm]{Corollary}
\newtheorem{prop}[thm]{Proposition}
\newtheorem{lem}[thm]{Lemma}
\newtheorem{conj}[thm]{Conjecture}

\theoremstyle{definition}

\newtheorem{exmp}[thm]{Example}

\theoremstyle{remark}
\newtheorem{rem}[thm]{Remark}

\newcommand{\Ext}{\operatorname{Ext}}
\newcommand{\ord}{\operatorname{ord}}

\newcommand{\C}{\mathfrak C}
\newcommand{\D}{\Delta}

\newcommand{\e}{\varepsilon}

\newcommand{\MTZF}{M_*(2)_{\mathbb Z/5}}

\makeatletter
\let\c@equation\c@thm
\makeatother
\numberwithin{equation}{section}

\title{Explicit modular forms from
the divided beta family}

\author{Donald M.\ Larson}

\begin{document}

\begin{abstract}
We compute modular forms known to
arise from the order 5 generators of the $5$-local Adams-Novikov spectral sequence 2-line, generalizing and contextualizing previous computations of M.\ Behrens and G.\
Laures. We exhibit analogous computations
at other primes and conjecture formulas for
some of the modular forms arising in this way
at arbitrary primes $\geq5$.
\end{abstract}

\maketitle

\tableofcontents

\section{Introduction}\label{Intro}
At an odd prime $p$, the 2-line
$\Ext^{2,*}$ of the Adams-Novikov
spectral sequence
\[
\Ext^{s,t}:=\Ext^{s,t}_{BP_*BP}(BP_*,BP_*)
\Rightarrow(\pi_{t-s}S^0)_{(p)}
\]
is generated additively by distinguished
classes $\beta_{i/j,k}\in\Ext^{2,
2i(p^2-1)-2j(p-1)}$
of order $p^k$ for certain ordered
triples of positive integers
$(i,j,k)$ \cite{MRW}.
These generators
are known collectively as the
$p$-primary divided beta family.
Behrens \cite{Beh:Cong}
has shown that, for
$p\geq5$, each
$\beta_{i/j,k}$ gives rise to
a modular form $f_{i/j,k}$
over the integers
satisfying conditions
formulated in terms of
$i$, $j$, $k$, and $p$,
including both a congruence
and non-congruence condition.
The proof that these
modular forms exist uses
homotopical properties of spectra
related to the $K(2)$-local
sphere \cite{Beh:Mod}
and ultimately does not explicitly identify the $f_{i/j,k}$.
The purpose of this paper is to
compute the modular forms
$f_{i/j,1}$ associated to order
$p$ divided beta family elements
$\beta_{i/j,1}$
in several cases, including
a complete computation of
such forms at the prime 5.

Behrens' theorem
\cite[Theorem~1.3]{Beh:Cong}
says the following in the case $k=1$.
Let $M_t$
be the space of weight $t$ modular forms
over $\mathbb Z$ for the full
modular group $\Gamma_0(1)=SL_2(\mathbb Z)$,
and for
a positive integer $N$, let $M_t(N)$
be the space of weight $t$ modular
forms over $\mathbb Z$
%of level $N$---that
%is to say, forms
for the congruence subgroup
$\Gamma_0(N)\subset \Gamma_0(1)$
\cite{Lang:Mod}.
Given an order $p$ divided beta
family element $\beta_{i/j,1}$
for $p\geq5$,
there exists a modular form
$f_{i/j,1}\in M_{i(p^2-1)}$ satisfying
the following conditions:
\begin{enumerate}
\item[(C1)] the Fourier expansion $f_{i/j,1}(q)\in\mathbb Z[[q]]$
is not congruent to 0 mod $p$,
\item[(C2)] the integer $12\cdot\ord_qf_{i/j,1}(q)$ is either
greater than $(p^2-1)i-(p-1)j$, or equal to $(p^2-1)i-(p-1)j-2$,
\item[(C3)] there does not exist $g\in M_t$ for $t<(p^2-1)i$ such that
$f_{i/j,1}(q)\equiv g(q)\mod p$, and
\item[(C4)]\label{C4} for every prime $\ell\neq p$,
there exists $g\in
M_{(p^2-1)i-(p-1)j}(\ell)$
such that $(L_{\ell}f_{i/j,1})(q)\equiv g(q)\mod
p$
\end{enumerate}
where %,
%given a weight $t$ and positive integer $N$,
$L_N:M_t\to M_t(N)$ is a linear operator
to be defined in Section \ref{ModularForms}.

There is ambiguity inherent in the computation
of these modular forms.  For example,
if a modular form $f$ satisfies
conditions (C1) through (C4) for given
$i$, $j$, and $p$, then so does
$cf+h$, where $c\in\mathbb Z$ with $(c,p)=1$
and $h$ is a modular form of the
same weight and $q$-order as $f$ such that
$h(q)\equiv 0\mod p$.  The precise extent to which
a given $f_{i/j,k}$ fails to be unique will be the
subject of future work.   The goal
here is to produce {\em at least one} candidate
for any $f_{i/j,1}$ that we study, and
an assertion of the form $f_{i/j,1}=f$
is understood to mean that $f$ is one of
many possible
choices for $f_{i/j,1}$ in $M_{i(p^2-1)}$.

Our results generalize sample
computer-aided computations of Behrens and Laures \cite{Beh:Lau}.
To state their results, and our
generalizations thereof, requires
additional notation.  In the divided
beta family, we write $\beta_{i/j}$ for
$\beta_{i/j,1}$, we write $\beta_i$ for $\beta_{i/1}$,
and we abbreviate the corresponding $f_{i/j,k}$ similarly.
In the
space of modular forms, $\D\in M_{12}$ is the
Ramanujan delta function and $E_4\in M_4$ is the
weight 4 Eisenstein series, 
whose definitions we will recall
in Section \ref{ModularForms}.
Behrens and Laures computed the following
modular forms at the prime 5:
\begin{align}\label{powersofD}
\begin{split}
f_1=\D^{2},\quad
f_2=\D^{4},\quad
f_3=\D^{6},\quad
f_4=\D^{8},\quad
f_{5/5}=\D^{10},
%f_{25/5,2}&=\D^{50},\\
\end{split}
\end{align}
and
\begin{equation}\label{FTFTN}
f_{25/29}=\D^{50}+4\D^{42}E_4^{24}
+3\D^{41}E_4^{27}.
\end{equation}
We shall see that the striking
pattern of powers of $\D$
in \eqref{powersofD}
does indeed persist at the
prime 5 and elsewhere.
%Although we will also
%verify that the modular form arising from
%the order 25 beta element
%$\beta_{25/5,2}$ is indeed $\D^{50}$,
%this is the only modular form we will
%compute for a beta element of order $p^k$
%for $k>1$.

The formula for $f_{25/29}$
in \eqref{FTFTN} is
the catalyst for our 5-primary computations.  
This is despite the fact that
the corresponding element
$\beta_{25/29}$ in homotopy theory is actually
not a member of the 5-primary divided beta family
(see Lemma \ref{betas}), but rather
can be properly interpreted as an element
of a different, closely related
spectral sequence (see \cite{Beh:Lau}).
The fact that $\D^{50}+4\D^{42}E_4^{24}
+3\D^{41}E_4^{27}$ is indeed
a modular form satisfying conditions
(C1) through (C4) in the case $i=25$,
$j=29$, and $p=5$, is instructive.
For example, while $\beta_{25/29}$
is not a 5-primary divided beta family element,
$\beta_{50/29}$ is, and
%one can check that the weight 120
%modular form $f_{50/29}$ cannot be
%$\D^{100}$.
the way $f_{25/29}$
is built as the sum of a power of $\D$ and two
``correction terms'' serves as a model for how to
compute $f_{50/29}$.  In fact,
it is by identifying appropriate
analogs of these ``correction terms''
that we are able to compute
all modular forms $f_{i/j}$ at the prime 5
that are not simply powers of $\D$.

Before stating the main theorem
let us define, for $n\geq2$ and $r\geq1$,
modular forms
\begin{align*}
C_{0,n,r}&=4r\D^{42\cdot5^{n-2}+2(r-1)5^n}E_4^{24\cdot5^{n-2}},\\
D_{0,n,r}&=3r\D^{41\cdot5^{n-2}+2(r-1)5^n}E_4^{27\cdot5^{n-2}}
\end{align*}
in $M_{24r\cdot5^{n}}$,
so that $C_{0,2,1}+D_{0,2,1}=4\D^{42}E_4^{24}
+3\D^{41}E_4^{27}$
is precisely the summand appearing
in \eqref{FTFTN}.
%the Behrens-Laures formula for $f_{25/29}$.
%\begin{align*}
%C_2&:=f_{25/29}-\D^{50}=4\D^{42}E_4^{24}
%+3\D^{41}E_4^{27},\\
%D_2&:=C_2-3\D^{41}E_4^{27}=4\D^{42}E_4^{24}
%\end{align*}
Moreover, for $n\geq3$, $1\leq m\leq n-2$, and $r\geq1$, define
\begin{align*}
C_{m,n,r}&=3r\D^{8\cdot5^{n-1}
	+2\cdot5^{n-m-2}+2(r-1)5^n}
E_4^{6\cdot5^{n-1}-6\cdot5^{n-m-2}},\\
%\quad D_n=\D^{8\cdot5^{n-1}+\cdot5^{n-m-2}}
%E_4^{6\cdot5^{n-1}-3\cdot5^{n-m-2}}
%\in M_{24\cdot5^n},\\
D_{m,n,r}&=r\D^{8\cdot5^{n-1}
	+5^{n-m-2}+2(r-1)5^n}
E_4^{6\cdot5^{n-1}-3\cdot5^{n-m-2}}
\end{align*}
in $M_{24r\cdot5^n}$.  Finally,
given a prime $p$ and
any integer $n\geq0$, define
\[
a_n=\begin{cases}
0,&\text{if }n=0,\\
p^n+p^{n-1}-1,&\text{if }n\geq1.
\end{cases}
\]

%% MAIN THEOREM %%
\begin{thm}\label{main_5}
Given a 5-primary divided beta family
element $\beta_{i/j}$ of order 5,
where $i=r\cdot5^n$ with $(r,5)=1$,
the corresponding modular form
is $f_{i/j}=\D^{2r\cdot5^n}$
except when the inequalities $r>1$ and $5^n+1\leq j\leq a_n$
hold simultaneously.  For such $r$ and $j$, there is a positive
integer $u$, $1\leq u\leq n-1$, such that
\[
5^n+5^{n-1}-5^{n-u}+1\leq j\leq5^n+5^{n-1}-5^{n-u-1}
\]
in which case
\[
f_{i/j}=\begin{cases}
\D^{2r\cdot5^n}+
\displaystyle\sum_{m=0}^{u-1}
(C_{m,n,r}+D_{m,n,r})
&\text{if}\quad j>5^n+5^{n-1}-5^{n-u}+2\cdot5^{n-u-1},\\
&\, \\
\D^{2r\cdot5^n}+\displaystyle\sum_{m=0}^{u-2}
(C_{m,n,r}+D_{m,n,r})+C_{u-1,n,r}
&\text{if}\quad j\leq5^n+5^{n-1}-5^{n-u}+2\cdot5^{n-u-1}.\\
\end{cases}
\]
\end{thm}
In the case $r>1$, Theorem \ref{main_5} shows that 
the modular forms $f_{i/j}$ can all be recovered
from $f_{i/a_n}$ by trimming off some number
of summands depending on the size of $j$.  
Example \ref{1250} below clearly exhibits this phenomenon.  
It is therefore worth noting that the formula
for $f_{i/a_n}$ in Theorem \ref{main_5} has a recursive
interpretation.  
\begin{cor}\label{main_5_cor}
For $n\geq1$ and $r\geq2$, $f_{r\cdot5^{n+1}/a_{n+1}}=
(f_{r\cdot5^n/a_n})^5
+C_{n-1,n+1,r}+D_{n-1,n+1,r}$.
\end{cor}
The formulas for a given $f_{i/a_n}$ 
obtained from Theorem \ref{main_5}
and Corollary \ref{main_5_cor}, respectively, will differ
in $M_{24i}$, but they will be equivalent
modulo 5.    
\begin{exmp}\label{preview}
To preview the concepts
and computational methods behind the
proof of Theorem \ref{main_5},
let us employ them to outline a
re-derivation of Equation \eqref{FTFTN}.
Since $f_{25/29}\in M_{600}$, it will
follow from Proposition \ref{basis}
and conditions (C1) and (C3)
that
\[
f_{25/29}=\D^{50}+\sum_{m=1}^{50}c_m\D^{50-m}
E_4^{3m}
\]
for integers $c_i\in\mathbb Z$.
The demand on $q$-order given by
condition (C2), together with Proposition
\ref{setup}(c),
will imply that $c_{10}=c_{11}=\cdots
=c_{50}=0$.
%, and so
%\[
%f_{25/29}=\D^{50}+\sum_{m=1}^{9}c_m\D^{50-m}
%E_4^{3m}.
%\]

To verify
condition (C4) at a prime $\ell\neq5$,
Proposition \ref{Serre} will imply
that it suffices to show
$L_{\ell}f_{25/29}$ is
divisible by $E_4^{29}$ in
$M_*(\ell)_{\mathbb Z/5}$.
Proposition \ref{rigidity} will imply
that it suffices to do this for
the case $\ell=2$.
In $M_*(2)_{\mathbb Z/5}$, we
shall see that $L_2\D^{50}$ is divisible
by $E_4^{25}$ by Lemma
\ref{L2ONDELTA}.
The $E_4$-divisibility
of $L_2\D^{50-m}E_4^{3m}$ for $1\leq m\leq9$
can be computed directly in a similar
manner using Lemma \ref{L2TERM}:\vskip.1truein
\setlength{\tabcolsep}{2pt}
\begin{center}
\begin{tabular}{c|c|c|c|c|c|c|c|c|c}
$m$&1&2&3&4&5&6&7&8&9\\\hline
term&$L_2\D^{49}E_4^3$&$L_2\D^{48}E_4^6$&
$L_2\D^{47}E_4^9$&$L_2\D^{46}E_4^{12}$&$L_2\D^{45}E_4^{15}$
&$L_2\D^{44}E_4^{18}$&$L_2\D^{43}E_4^{21}$&$L_2\D^{42}E_4^{24}$
&$L_2\D^{41}E_4^{27}$\\\hline
$E_4$-div.&$E_4^3$&$E_4^7$&$E_4^9$&$E_4^{13}$&$E_4^{15}$&
$E_4^{19}$&$E_4^{21}$&$E_4^{25}$&$E_4^{27}$
\end{tabular}
\end{center}\vskip.1truein
Thus, for $L_2f_{25/29}$ to have the required
$E_4$-divisibility,
we much choose $c_1=c_2=\cdots=c_7=0$.
From there, setting $c_8=4$ makes the term
\[
L_2(\D^{50}+c_8\D^{42}E_4^{24})
\]
divisible by $E_4^{27}$ in
$M_*(2)_{\mathbb Z/5}$, while all other choices
for $c_8$ (modulo 5) keep its $E_4$-divisibility
at $E_4^{25}$.  Subsequently setting
$c_9=3$ makes
\[
L_2(\D^{50}+c_8\D^{42}E_4^{24}
+c_9\D^{41}E_4^{27})
\]
divisible by $E_4^{29}$ in
$M_*(2)_{\mathbb Z/5}$, and no other choice
of $c_9$ (modulo 5) accomplishes this.  
\end{exmp}
\begin{exmp}\label{1250}
Consider the conclusion of
Theorem \ref{main_5} in the
case $n=4$ and $r=2$.
The order 5 divided beta family elements $\beta_{i/j}$
with $i=2\cdot5^4=1250$ are
\[
\{\beta_{1250/j}:1\leq j\leq749\text{ and }
j\neq5,10,15,\ldots,125\}
\]
(see Lemma \ref{betas}).  By Theorem \ref{main_5},
$f_{1250/j}=\D^{2500}$ except when $626\leq j\leq749$,
in which case $f_{1250/j}=\D^{2500}+\C$ where
\[
\C=\begin{cases}
3\D^{2300}E_4^{600}+
\D^{2275}E_4^{675}+
\D^{2260}E_4^{720}+
2\D^{2255}E_4^{735}+
\D^{2252}E_4^{744}+
2\D^{2251}E_4^{747}
&\text{if }\quad j=748,749,\\
3\D^{2300}E_4^{600}+
\D^{2275}E_4^{675}+
\D^{2260}E_4^{720}+
2\D^{2255}E_4^{735}+
\D^{2252}E_4^{744}&\text{if }\quad j=746,747,\\
3\D^{2300}E_4^{600}+
\D^{2275}E_4^{675}+
\D^{2260}E_4^{720}+
2\D^{2255}E_4^{735}&\text{if }\quad 736\leq j\leq745,\\
3\D^{2300}E_4^{600}+
\D^{2275}E_4^{675}+
\D^{2260}E_4^{720}&\text{if }\quad 726\leq j\leq735,\\
3\D^{2300}E_4^{600}+
\D^{2275}E_4^{675}&\text{if }\quad 676\leq j\leq725,\\
3\D^{2300}E_4^{600}&\text{if }\quad 626\leq j\leq675.
\end{cases}
\]
In each of the six cases,
the term $\C$ raises the $E_4$-divisibility
of $L_2f_{1250/j}$ in
$M_*(2)_{\mathbb Z/5}$
as required by
condition (C4), while keeping
the $q$-order of $f_{1250/j}$
sufficiently large as required
by condition (C2).
% In particular,
%the addition of
%$\D^{2252}E_4^{744}+
%2\D^{2251}E_4^{747}$ raises the
%$E_4$-divisibility from 745
%to 749 in this example, just as
%the addition of
%$4\D^{42}E_4^{24}
%+3\D^{41}E_4^{27}$
%raises the $E_4$-divisibility
%from 25 to 29 in Example
%\ref{preview}.
\end{exmp}
\begin{rem}
In addition to the modular forms
in Equations \eqref{powersofD} and
\eqref{FTFTN},
Behrens and Laures computed
$f_{25/5,2}=\D^{50}$, where $\beta_{25/5,2}$
is the first order 25 element one encounters
in the 5-primary divided beta family.  
Computation of the modular forms $f_{i/j,k}$
for $k>1$ will be the subject of a future paper.
\end{rem}

This paper is structured as follows.
In Section \ref{Divided} we recall
how to enumerate the elements of order $p$
in the $p$-primary divided beta family.
In Section \ref{ModularForms} we list results
from the theory of modular forms relevant
to our computations.  Section \ref{at5}
is the technical heart of the paper,
where we compute at the prime 5 and prove
Theorem \ref{main_5}.
In Section \ref{atotherprimes}
we exhibit our computational methods
at the primes 7, 11, 13, and 677,
and for arbitrary primes $p\geq5$ we
conjecture which of the modular forms arising from Behrens' theorem are simply powers
of the Ramanujan delta function.
\section{The divided beta family}
\label{Divided}
In this section we show how to
enumerate the order $p$ elements
$\beta_{i/j}$
of the $p$-primary divided beta family.
Given $p$, let $a_n$ be defined
as in Section \ref{Intro}.
If we write $i=rp^n$
where $(r,p)=1$, then
\cite[Theorem~2.6]{MRW} implies that the
entire $p$-primary
divided beta family comprises the elements
$\beta_{i/j,k}$ for ordered triples of positive
integers $(i,j,k)$ subject to the following
rules:
\begin{enumerate}[i)]
\item if $r=1$, then $1\leq j\leq p^n$,
\item $p^{k-1}|j\leq a_{n-k+1}$,
\item if $p^k|j$, then $j>a_{n-k}$.
\end{enumerate}
The order $p$ divided beta family elements $\beta_{i/j}$
are therefore characterized as follows.
\begin{lem}\label{betas}
The divided beta family elements of
order $p$ of the form $\beta_{p^n/j}$
are
\[
\{\beta_{p^n/j}:1\leq j\leq p^n\text{ and }j\neq p,2p,\ldots, %a_{n-1}-p+1\}.
a_{n-2}p\}.
\]
and those of the form $\beta_{rp^n/j}$
where $r>1$ and $(r,p)=1$ are
\[
\{\beta_{rp^n/j}:1\leq j\leq a_n\text{ and }j\neq p,2p,\ldots, %a_{n-1}-p+1\}.
a_{n-2}p\}.
\]
\end{lem}
\begin{proof}
Consider first the case $r=1$.
Rule i) allows $j$ to range between
1 and $p^n$.  Since $k=1$,
rule ii) does not apply
any additional constraints.
Rule iii) says that $j$ cannot
be both a multiple of $p$ and
$\leq a_{n-1}=p^{n-1}+p^{n-2}-1$.
Therefore, we must insist that
\[
j\neq p,2p,\ldots,p^{n-1}-p^{n-2}-5=a_{n-2}p.
\]

In the case $r>1$, rule i)
does not apply and rule ii)
allows $j$ to range between
1 and $a_n$.  Rule iii)
disallows the same values of $j$
as in the case $r=1$.
\end{proof}
\begin{exmp}\label{exmp_i}
Given $p$, the element $\beta_i=\beta_{i/1,1}$ is a divided
beta family element of order $p$ for any $i\geq1$.
\end{exmp}
\begin{exmp}\label{exmp_ii}
Suppose $p=5$.  The divided beta family elements
of the form $\beta_{5r/j}$ are
\[
\beta_{5r/5},\quad
\beta_{5r/4},\quad
\beta_{5r/3},\quad
\beta_{5r/2},\quad
\beta_{5r/1}=\beta_{5r}.
\]
Those of the form
$\beta_{25/j}$ are
\[
\beta_{25/25},\quad \beta_{25/24},
\quad\ldots,\quad\beta_{25/6},\quad\beta_{25/4},
\quad\ldots,\quad\beta_{25/1}=\beta_{25}
\]
%all of which have order 5, along with $\beta_{25/5,2}$
%of order 25.  In particular,
%there is no divided beta family element
%$\beta_{25/29}$.
%On the other hand,
while those of the form $\beta_{25r/j}$ with $r>1$
are
\[
\beta_{25r/29},\quad \beta_{25r/28},\quad\ldots,\quad
\beta_{25r/6},\quad\beta_{25r/4},\quad
\ldots,\quad\beta_{25r/1}=\beta_{25r}.
\]
%all of which have order 5, along with $\beta_{25r/5,2}$
%of order 25.
In particular,
as noted in Section \ref{Intro},
$\beta_{50/29}$ is a 5-primary
divided beta family element, while $\beta_{25/29}$
is not.
%Note that the highest possible
%value of $j$ in these latter cases is $a_2=25+5-1=29$.
\end{exmp}
%\begin{exmp}
%Suppose $p=5$ once again.  The first
%divided beta family element of order 125 one encounters
%is $\beta_{625/25,3}$.  The first one of
%order 625 is $\beta_{15625/125,4}$.
%\end{exmp}
\section{Modular forms}\label{ModularForms}

In this section we record
facts from the theory of modular forms
required for our computations.
Much of this can be found in
\cite{Lang:Mod} unless otherwise noted.

If $t\geq4$ is an even integer
and $q=e^{2\pi iz}$, the
{\em weight $t$ Eisenstein series}
$E_t\in M_t$ is given by the formula
\[
E_t(z)=\dfrac12\sum_{(m,n)=1}\dfrac1{(mz+n)^t}
=1-\dfrac{2t}{B_t}\sum_{n=1}^{\infty}
\sigma_{t-1}(n)q^n
\]
where $B_t$ is the $t$-th Bernoulli number
and $\sigma_{t-1}(n)$ is the sum of the $(t-1)$st
powers of the divisors of $n$.
The {\em Ramanujan Delta function}
$\D$ is a modular form of weight 12
given by the formula
\[
\D(z)=\dfrac{E_4^3(z)-E_6^2(z)}{1728}
=\sum_{n=1}^{\infty}\tau(n)q^n=q+\cdots.
\]
The graded (by weight) ring $M_*$ of
all modular forms over $\mathbb Z$
for the full modular group is,
by a result of Deligne \cite[Proposition~6.1]{Del:Ell},
\[
M_*=\dfrac{\mathbb Z[E_4,E_6,\D]}{1728\D=E_4^3-E_6^2}.
\]
\begin{prop}{\cite[Theorem~10.4.3]{Lang:Mod}}\label{basis}
If $t\equiv0\mod 4$, then
$\{\D^aE_4^b:0<a,b\in\mathbb Z, 12a+4b=t\}$
is a $\mathbb Z$-basis for $M_t$.
\end{prop}

Let $M_*(2)[1/2]$ denote the
graded ring of modular forms
for $\Gamma_0(2)$ with 2 inverted.
The structure of this ring
is well-known (see, e.g.,
\cite[Appendix~I]{HBJ}).
\begin{prop}\label{leveltwo}
There exist modular forms
$\delta\in M_2(2)[1/2]$
and $\e\in M_4(2)[1/2]$
with
\begin{align*}
\delta(q)&=4^{-1}+6q+6q^2+
24q^3+6q^4+36q^5+24q^6+\cdots,\\
\e(q)&=16^{-1}-q+7q^2-28q^3
+71q^4-126q^5+196q^6-\cdots,
\end{align*}
and such that $M_*(2)[1/2]
=\mathbb Z[1/2][\delta,\e]$.
\end{prop}
Given an integer $N\geq1$,
we define two linear operators
on modular forms:
the first is
\begin{align*}
\iota_N:M_*&\to M_*(N)\\
f&\mapsto f
\end{align*}
obtained by regarding $f\in M_*$
as a modular form for $\Gamma_0(N)$;
the second is the {\em Verschiebung}
\begin{align*}
V_N:M_*&\to M_*(N)\\
f(q)&\mapsto f(q^N)
\end{align*}
which satisfies $V_N(fg)=V_N(f)V_N(g)$
since $q\mapsto q^N$ is a ring
endomorphism of $\mathbb Z[[q]]$.
In particular, we will often write
$V_Nf^m$ for $1\leq m\in\mathbb Z$, which is unambiguous since
$V_N(f^m)=(V_Nf)^m$.
The linear operator
appearing in condition (C4) of Behrens'
theorem is
\[
L_N=V_N-\iota_N.
\]
We will not distinguish between
$f$ and $\iota_Nf$ and so
we will write $L_Nf=V_Nf-f$.  
%\begin{enumerate}
%\item[(C4)] For every prime $\ell\neq p$,
%there exists $g\in
%M_{(p^2-1)i-(p-1)j}(\ell)$
%such that $L_{\ell}f_{i/j,k}(q)\equiv g(q)\mod
%p^k$.
%\end{enumerate}

It will be
convenient to define
\[
\mu=\delta^2-\e\in M_4(2)[1/2]
\]
and direct computations
in $M_*(2)[1/2]$
yield the identities
\begin{align}
E_4&=64\mu+16\e%\quad V_2E_4=4\mu+16\e
\label{E4identity},\\
V_2E_4&=4\mu+16\e
\label{vE4identity},\\
%E_6&=-512\delta\mu+64\delta\e,\quad
%V_2E_6=-8\delta\mu+64\delta\e,\\
\D&=64\mu\e^2\label{deltaidentity},\\
V_2\D&=\mu^2\e\label{vdeltaidentity}.
\end{align}

Let $M_t(N)_{\mathbb Z/p}$ be
the weight $t$ modular forms
for $\Gamma_0(N)$ defined over
$\mathbb Z/p$.  As long as
$N$ is invertible in $\mathbb Z/p$,
\[
M_t(N)_{\mathbb Z/p}=M_t(N)\otimes
\mathbb Z/p
\]
(see \cite[Equation~(1.1)]{Beh:Cong}).
We will not draw a notational distinction
between $f\in M_t(N)$ and its mod $p$
reduction $f\in M_t(N)_{\mathbb Z/p}$
as the meaning will always be clear from
the context.    

The following result of Serre
is a generalization of
the well known congruence
$E_{p-1}(q)\equiv 1\mod p$, and
dictates exactly when congruences
between modular forms
of different weights modulo
$p$ can occur (see \cite[Theorem~10.2]{Beh:Cong}
or \cite[Corollary~4.4.2]{Katz:padic}).
\begin{prop}[Serre]\label{Serre}
Let $f_1\in M_{t_1}(N)_{\mathbb Z/p}$
and
$f_2\in M_{t_2}(N)_{\mathbb Z/p}$
with $t_1<t_2$.  Then
\[
f_1(q)=f_2(q)\in\mathbb Z/p[[q]]
\]
(that is, $f_1\equiv f_2\mod p$)
if and only if
$t_1\equiv t_2\mod (p-1)$
and $f_2=E_{p-1}^{\frac{t_2-t_1}
	{p-1}}f_1$.
\end{prop}

Finally,
the following rigidity result
implies that checking condition
(C4) for a given $f_{i/j}$
often reduces to a computation
with $\Gamma_0(\ell)$ modular forms 
at a {\em single} prime $\ell\neq p$
(see \cite[Theorem~1.5]{Beh:Cong}).
\begin{prop}[Behrens]\label{rigidity}
	If the prime $\ell_0$ is a topological
	generator of $\mathbb Z_p^{\times}$
	and $f$ is a modular form
	of weight $t\equiv (p-1)$
	satisfying conditions (C1) through (C3),
	as well as condition (C4) for
	$\ell=\ell_0$, then $f$ satisfies
	condition (C4) for all primes
	$\ell\neq p$.
\end{prop}

\section{Computations at the prime 5}
\label{at5}
In this section we fix $p=5$
and prove Theorem \ref{main_5}.
The first goal is to work toward a proof
that the modular forms identified as
$f_{i/j}$ in Theorem \ref{main_5} satisfy
conditions (C1), (C2), and (C3) of Behrens'
theorem.
\begin{prop}\label{setup}
Suppose $f\in M_t$ with $t\equiv0\mod 12$.
\begin{enumerate}[(a)]
\item There exist integers $c_0,\ldots,c_{t/12}$
such that
\[
f=c_0\D^{t/12}+c_1\D^{(t/12)-1}E_4^3+\cdots+	c_{(t/12)-1}\D E_4^{(t/4)-3}+c_{t/12}E_4^{t/4}.
\]
\item If $c_0\neq0\mod 5$ then $f(q)\not\equiv0\mod 5$
and $f$ is not
divisible by $E_4$ in $(M_*)_{\mathbb Z/5}$.
\item If $c_m\neq0$ and
$c_{m+1},c_{m+2},\ldots,c_{t/12}=0$
for some $m$, $0\leq m\leq t/12$, then
\[	
\ord_qf=\dfrac t{12}-m.
\]
\end{enumerate}
\end{prop}
\begin{proof}
Part (a) is a special case of Proposition
\ref{basis}.  Part (b) follows from the fact
that the integral basis for $M_t$ given by
Proposition \ref{basis} remains a basis
for the vector space $(M_t)_{\mathbb Z/5}$
(see, e.g., Section 1 of \cite{Jochnowitz}).
Since $\D(q)=q+\cdots$ and $E_4(q)=1+\cdots$,
the value of $\ord_q(f)$ is the smallest power of
$\D$ appearing in the expansion of $f$ given
in Part (a), and Part (c) follows.
\end{proof}
\begin{cor}\label{C1C3}
The modular forms identified as $f_{i/j}$
in Theorem \ref{main_5}
satisfy conditions (C1) and (C3) of Behrens' theorem.
\end{cor}
\begin{proof}
The modular form $f_{i/j}$ has weight
$24i$ and so has an expansion
as in Proposition \ref{setup}(a).
The formulas given in
Theorem \ref{main_5} show that each such expansion
has leading coefficient $c_0=1$ by construction.
Thus $f_{i/j}(q)\neq0\mod 5$ by Proposition \ref{setup}(b),
and so $f_{i/j}$ satisfies condition (C1).
Condition (C3) is also satisfied by
Propositions \ref{setup}(b) and \ref{Serre}.
\end{proof}
\begin{lem}\label{qineq}
At the prime 5, condition (C2) is equivalent
to $\ord_qf_{i/j}>2i-\dfrac j3$.
%First, the equality
%in condition (C2) is impossible when $p=5$
%because $24i-4j-2$ is never divisible by
%12.  Therefore, the condition in this case says
\end{lem}
\begin{proof}
At the prime 5, condition (C2) says that either
$\ord_qf_{i/j}>(24i-4j)/12$ or $\ord_qf_{i/j}=(24i-4j-2)/12$.
But $24i-4j-2$ is never divisible by 12 for integer values
of $i$ and $j$, so condition (C2) reduces to the inequality
only.
\end{proof}
\begin{prop}\label{C2ONLY}
The modular forms $f_{i/j}$ as identified
in Theorem 1.1 all satisfy condition (C2) of Behrens' theorem.
\end{prop}
\begin{proof}
Suppose $i=r\cdot 5^n$ with $(r,5)=1$.  Consider first the case $j\leq p^n$,
so that $f_{r\cdot5^n/j}=\D^{2r\cdot5^n}$.  By
Lemma \ref{qineq}, verifying condition (C2)
is equivalent to verifying the inequality
\begin{equation}\label{checkineq}
\ord_q\D^{2r\cdot5^n}>2r\cdot5^n-\dfrac j3.
\end{equation}
but $\ord_q\D^{2r\cdot5^n}=2r\cdot5^n$
by Proposition \ref{setup}(c),
so \eqref{checkineq} clearly holds for $j\geq1$.

Next, suppose $j>p^n$.  Let $u$ be the
positive integer between 1 and $n-1$ such that
\[
5^n+5^{n-1}-5^{n-u}+1\leq j\leq
5^n+5^{n-1}-5^{n-u-1}.
\]
Assume $j>5^n+5^{n-1}-5^{n-u}+2\cdot5^{n-u-1}$.
By the formulas given in Theorem \ref{main_5},
we must verify condition (C2) for
\[
f_{r\cdot5^n/j}=\D^{2r\cdot5^n}+\sum_{m=0}^{u-1}
(C_{m,n,r}+D_{m,n,r}),
\]
a form whose $q$-order is the power
of $\D$ occurring in $D_{u-1,n,r}$ by Proposition 
\ref{setup}(c).
Therefore, by Lemma \ref{qineq}, we must verify
the inequality
\begin{equation}\label{checkineq2}
8\cdot5^{n-1}+5^{n-u-1}+2(r-1)5^n>2r\cdot5^n-\dfrac j3.
\end{equation}
It suffices to verify \eqref{checkineq2} for the
smallest possible value of $j$, which in this case
is \[j=5^n+5^{n-1}-5^{n-u}+2\cdot5^{n-u-1}+1.\]
At this value of $j$, \eqref{checkineq2} becomes
\[
8\cdot5^{n-1}+5^{n-u-1}+2(r-1)5^n>2r\cdot5^n-
\dfrac{5^n+5^{n-1}-5^{n-u}+2\cdot5^{n-u-1}+1}3
\]
which is equivalent to $1/3>0$. So \eqref{checkineq2}
holds.

Continuing with the case $j>p^n$, we now assume
$j\leq5^n+5^{n-1}-5^{n-u}+2\cdot5^{n-u-1}$.
By the formulas given in Theorem \ref{main_5}, we must verify
condition (C2) for
\[
f_{r\cdot5^n/j}=\D^{2r\cdot5^n}+\sum_{m=0}^{u-2}
(C_{m,n,r}+D_{m,n,r})+C_{u-1,n,r}
\]
whose $q$-order is the power of $\D$ occurring
in $C_{u-1,n,r}$ by Proposition \ref{setup}(c).  Therefore,
by Lemma \ref{qineq},
the inequality we must verify is
\begin{equation}\label{checkineq3}
8\cdot5^{n-1}+2\cdot5^{n-u-1}+2(r-1)5^n>2r\cdot5^n-\dfrac j3.
\end{equation}
The smallest possible value of $j$ is now
$j=5^n+5^{n-1}-5^{n-u}+1$, so it suffices to verify
\[
8\cdot5^{n-1}+2\cdot5^{n-u-1}+2(r-1)5^n>2r\cdot5^n-
\dfrac{5^n+5^{n-1}-5^{n-u}+1}3
\]
which is equivalent to
\[
\dfrac{5^{n-u-1}+1}3>0.
\]
So \eqref{checkineq3} holds, and we have shown
condition (C2) is satisfied in all cases.
\end{proof}
\begin{lem}\label{rigidlem}
If a modular form $f_{i/j}$
identified in Theorem \ref{main_5} satisfies
conditions (C1) through (C3), then it satisfies 
condition (C4) if and only if $L_2f_{i/j}$ is divisible by
$E_4^j$ in $M_*(2)_{\mathbb Z/5}$.
%Because 2 is a topological
%generator of $\mathbb Z_5^{\times}$,
%Proposition \ref{rigidity} implies
%that we need only check condition
%(C4) at level $\ell=2$.
\end{lem}
\begin{proof}
Since 2 is a topological generator of $\mathbb Z_5^{\times}$,
checking $f_{i/j}$ satisfies condition (C4) is equivalent to showing there exists $g\in M_{24i-4j}(2)$ such
that $(L_2f_{i/j})(q)\equiv g(q)\mod 5$
by Proposition \ref{rigidity}.  The lemma then follows
from Proposition \ref{Serre}.  
%shows this is equivalent to showing $L_2f_{i/j}$
%is divisible by $E_4^j$ in $M_*(2)_{\mathbb Z/5}$.
\end{proof}

The remaining goal of this section is
to complete the proof of Theorem \ref{main_5}
by showing that the modular forms $f_{i/j}$
as identified in the theorem satisfy condition (C4).
%Lemma \ref{rigidlem} shows that we need only
%study $L_2f_{i/j}\in M_*(2)$, and Proposition \ref{Serre}
%shows more precisely that the key is to compute the $E_4$-divisibility of $L_2f_{i/j}$ when viewed as an element of
%$\MTZF$.
To begin, we establish a
method for computing the $E_4$-divisibility
required by Lemma \ref{rigidlem}.
\begin{prop}\label{method}
If $f\in M_t$, then $L_2f\in M_t(2)_{\mathbb Z/5}$
is expressible as a homogeneous element
of $\mathbb Z/5[\mu,\e]$ of degree $t/4$,
and if $y=4\mu/\e+1$, then
%$x=\mu/\e$ and $y=4x+1$,
\begin{equation}\label{inhomog}
L_2f=\e^{t/4}P(y)
\end{equation}
for an inhomogeneous polynomial $P(y)\in\mathbb Z/5[y]$.
Moreover, $L_2f$ is divisible by $E_4^j$ in
$\MTZF$ if and only if $P(y)=O(y^j)$
(that is, $P(y)$ is divisible by $y^j$)
in $\mathbb Z/5[y]$.
\end{prop}
\begin{proof}
By Proposition \ref{basis}, $f\in M_t$
is a $\mathbb Z$-linear combination of terms
of the form $\D^a E_4^b$ for non-negative integers
$a$, $b$ with $12a+4b=t$.
From the basic properties of $L_N$ outlined
in Section \ref{Intro}, as well as
Equations \eqref{E4identity}-\eqref{vdeltaidentity},
it follows that
\[
L_2(\D^aE_4^b)=V_2\D^aV_2E_4^b-\D^aE_4^b=
(\mu^2\e)^a(4\mu+16\e)^b-(64\mu\e^2)^a(64\mu+16\e)^b
\in M_t(2)
\]
which is a homogeneous polynomial in $\mu$ and $\e$
over the integers of degree $3a+b=t/4$.  Passing
to $M_t(2)_{\mathbb Z/5}=M_t(2)\otimes\mathbb Z/5$
yields an expression of $L_2(\D^aE_4^b)$
as a homogeneous polynomial of the same degree
over $\mathbb Z/5$.  Thus, $L_2f\in M_t(2)_{\mathbb Z/5}$ is a
sum of such homogeneous polynomials by the linearity
of $L_2$.  If we put $x=\mu/\e$, then $L_2f=\e^{t/4}P(x)$
where $P(x)\in\mathbb Z/5[x]$, and making the change
of variable $x=4y+1$ (so that $y=4x+1$) yields
%$L_2f=\e^{t/4}P(y)$.
Equation \eqref{inhomog}.

By Equation \eqref{E4identity}, $L_2f$ is divisible by $E_4^j$
in $\MTZF$ if and only if the corresponding homogeneous
element of $\mathbb Z/5[\mu,\e]$ is divisible by
$(4\mu+\e)^j$.  But this is equivalent to
$P(y)$ being divisible by $y^j$ since
\[
(4\mu+\e)^j=\e^j(4x+1)^j=\e^j y^j.
\]
\end{proof}
%We now establish which modular
%forms $f_{i/j}$ may be taken to be
%a power of $\D$.
%\begin{prop}
%A divided beta family element
%$\beta_{r\cdot5^n/j}$ such that
%$1\leq j\leq 5^n$
%has corresponding modular form
%$f_{r\cdot5^n/j}=\D^{2r\cdot5^n}$.
%\end{prop}
The next four lemmas produce equations
of the form \eqref{inhomog} for various
modular forms $f$.
\begin{lem}\label{L2TERM} For $0\leq a,b\in\mathbb Z$,
$L_2(\D^a E_4^b)=\e^{3a+b}y^b
((4y+1)^{2a}-(-1)^a(4y+1)^a)$
in $\MTZF$.
\end{lem}
\begin{proof}
Equation \eqref{vE4identity} implies that
$V_2E_4=E_4$ in $\MTZF$.  Therefore,
using the notation from the proof
of Proposition \ref{method},
\begin{align*}
L_2(\D^a E_4^b)&=V_2\D^aV_2E_4^b-\D^aE_4^b\\
&=E_4^b(V_2\D^a-\D^a)\\
&=(4\mu+\e)^b(\mu^{2a}\e^a-(-1)^a\mu^a\e^{2a})\\
&=\e^{3a+b}(4x+1)^b(x^{2a}-(-1)^ax^a)\\
&=\e^{3a+b}y^b((4y-1)^{2a}-(-1)^a(4y+1)^a)
\end{align*}
in $\MTZF$.
\end{proof}
\begin{lem}\label{L2ONDELTA}
For integers $r\geq1$ and $n\geq0$,
$L_2\D^{2r\cdot5^n}=\e^{6r\cdot5^n}
(3ry^{5^n}+O(y^{a_n}))$ in $\MTZF$.
\end{lem}
\begin{proof}
By Lemma \ref{L2TERM},
\begin{align*}
L_2(\D^{2r\cdot5^n})&=\e^{6\cdot5^n}
((4y+1)^{4r\cdot5^n}-(4y+1)^{2r\cdot5^n})\\
&=\e^{6\cdot5^n}
((4y^{5^n}+1)^{4r}-(4y^{5^n}+1)^{2r})\\
&=\e^{6\cdot5^n}(1+4r\cdot4y^{5^n}+O(y^{2\cdot5^n})
-(1+2r\cdot4y^{5^n}+O(y^{2\cdot5^n})))\\
&=\e^{6\cdot5^n}(3ry^{5^n}+O(y^{a_n})).
\end{align*}
\end{proof}
%\newpage
\begin{lem}\label{CDEXTREME}
For integers $n\geq2$ and $r\geq1$, %In $\in\MTZF$,
\begin{align*}
L_2C_{0,n,r}&=r\e^{6r\cdot5^n}(2y^{5^n}
+4y^{27\cdot5^{n-2}}+4y^{28\cdot5^{n-2}}
+3y^{29\cdot5^{n-2}}+O(y^{a_n})),\\
L_2D_{0,n,r}&=r\e^{6r\cdot5^n}(
y^{27\cdot5^{n-2}}+y^{28\cdot5^{n-2}}
+3y^{29\cdot5^{n-2}}
+O(y^{a_n})),\quad\text{and}\\
L_2D_{n-2,n,r}&=r\e^{6r\cdot5^n}
%(2y^{29\cdot5^{n-2}+2}
%+2y^{29\cdot5^{n-2}+3}
(2y^{6\cdot5^{n-1}-3}
+2y^{6\cdot5^{n-1}-2}
+O(y^{a_n})).
\end{align*}
in $\MTZF$.
%$L_2C_{0,n,r}=r\e^{6r\cdot5^n}(2y^{5^n}
%+4y^{27\cdot5^{n-2}}+4y^{28\cdot5^{n-2}}
%+3y^{29\cdot5^{n-2}}+O(y^{a_n}))$.
\end{lem}
\begin{proof}
To begin, we compute $L_2C_{0,n,r}$.
By Lemma \ref{L2TERM},
\begin{align}\label{C01}
\begin{split}
L_2C_{0,n,r}&=4rL_2(
\D^{42\cdot5^{n-2}+2(r-1)5^n}E_4^{24\cdot5^{n-2}}
)\\
&=4r\e^{6r\cdot5^n}y^{24\cdot5^{n-2}}
((4y+1)^{5^{n-2}(100r-16)}-(4y+1)^{5^{n-2}(50r-8)})\\
&=4r\e^{6r\cdot5^n}y^{24\cdot5^{n-2}}
((4y^{5^{n-2}}+1)^{25(4r-1)+5+4}-
(4y^{5^{n-2}}+1)^{25(2r-1)+5(3)+2}).
\end{split}
\end{align}
In $\mathbb Z/5[y]$,
\begin{align*}
(4y^{5^{n-2}}+1)^{25(4r-1)+5+4}&=
(4y^{5^n}+1)^{4r-1}(4y^{5^{n-1}}+1)
(4y^{5^{n-2}}+1)^4\\
&=(4y^{5^n}+1)^{4r-1}(4y^{5^{n-1}}+1)
(1+y^{5^{n-2}}+y^{2\cdot5^{n-2}}+y^{3\cdot5^{n-2}}+
y^{4\cdot5^{n-2}})\\
&=(4y^{5^n}+1)^{4r-1}
(1+y^{5^{n-2}}+y^{2\cdot5^{n-2}}+y^{3\cdot5^{n-2}}+
y^{4\cdot5^{n-2}}+4y^{5^{n-1}}+O(y^{6\cdot5^{n-2}}))\\
&=1+y^{5^{n-2}}+y^{2\cdot5^{n-2}}+y^{3\cdot5^{n-2}}+
y^{4\cdot5^{n-2}}+4y^{5^{n-1}}+O(y^{6\cdot5^{n-2}})
\end{align*}
and
\begin{align*}
(4y^{5^{n-2}}+1)^{25(2r-1)+5(3)+2}&=
(4y^{5^n}+1)^{2r-1}(4y^{5^{n-1}}+1)^3(4y^{5^{n-2}}+1)^2\\
&=(4y^{5^n}+1)^{2r-1}(1+2y^{5^{n-1}}+3y^{2\cdot5^{n-1}}
+4y^{3\cdot5^{n-1}})
(1+3y^{5^{n-2}}+y^{2\cdot5^{n-2}})\\
&=(4y^{5^n}+1)^{2r-1}(1+3y^{5^{n-2}}+y^{2\cdot5^{n-2}}
+2y^{5^{n-1}}+O(y^{6\cdot5^{n-2}}))\\
&=1+3y^{5^{n-2}}+y^{2\cdot5^{n-2}}
+2y^{5^{n-1}}+O(y^{6\cdot5^{n-2}})
\end{align*}
which, when combined with \eqref{C01}, yield
\begin{align*}
L_2C_{0,n,r}&=4r\e^{6r\cdot5^n}y^{24\cdot5^{n-2}}
(3y^{5^{n-2}}+y^{3\cdot5^{n-2}}+y^{4\cdot5^{n-2}}
+2y^{5^{n-1}}+O(y^{6\cdot5^{n-2}}))\\
&=r\e^{6r\cdot5^n}(2y^{5^n}
+4y^{27\cdot5^{n-2}}+4y^{28\cdot5^{n-2}}
+3y^{29\cdot5^{n-2}}+O(y^{a_n}))
\end{align*}
in $\MTZF$. % This gives Equation...

Next, we compute $L_2D_{0,n,r}$.  By Lemma \ref{L2TERM},
\begin{align}\label{D01}
\begin{split}
L_2D_{0,n,r}&=3rL_2(\D^{41\cdot5^{n-2}+2(r-1)5^n}
E_4^{27\cdot5^{n-2}})\\
&=3r\e^{6r\cdot5^n}y^{27\cdot5^{n-2}}
((4y+1)^{5^{n-2}(100r-18)}+(4y+1)^{5^{n-2}(50r-9)})\\
&=3r\e^{6r\cdot5^n}y^{27\cdot5^{n-2}}
((4y^{5^{n-2}}+1)^{25(4r-1)+5+2}
+(4y^{5^{n-2}}+1)^{25(2r-1)+5(3)+1}).
\end{split}
\end{align}
In $\mathbb Z/5[y]$,
\begin{align*}
(4y^{5^{n-2}}+1)^{25(4r-1)+5+2}&=(4y^{5^n}+1)^{4r-1}
(4y^{5^{n-1}}+1)(4y^{5^{n-2}}+1)^2\\
&=(4y^{5^n}+1)^{4r-1}(4y^{5^{n-1}}+1)
(1+3y^{5^{n-2}}+y^{2\cdot5^{n-2}})\\
&=1+3y^{5^{n-2}}+y^{2\cdot5^{n-2}}+O(y^{3\cdot5^{n-2}})
\end{align*}
and
\begin{align*}
(4y^{5^{n-2}}+1)^{25(2r-1)+5(3)+1}&=(4y^{5^n}+1)^{2r-1}
(4y^{5^{n-1}}+1)^3(4y^{5^{n-2}}+1)\\
&=1+4y^{5^{n-2}}+O(y^{3\cdot5^{n-2}})
\end{align*}
which, when combined with \eqref{D01}, yield
\begin{align*}
L_2D_{0,n,r}&=3r\e^{6\cdot5^n}y^{27\cdot5^{n-2}}
(2+2y^{5^{n-2}}+y^{2\cdot5^{n-2}}+O(y^{3\cdot5^{n-2}}))\\
&=r\e^{6r\cdot5^n}(
y^{27\cdot5^{n-2}}+y^{28\cdot5^{n-2}}+3y^{29\cdot5^{n-2}}
+O(y^{a_n}))
\end{align*}
in $\MTZF$.  %This gives Equation...

Finally, we compute $L_2D_{n-2,n,r}$.
By Lemma \ref{L2TERM},
\begin{align*}
L_2D_{n-2,n,r}&=rL_2(\D^{8\cdot5^{n-1}+1+2(r-1)5^n}
E_4^{6\cdot5^{n-1}-3})\\
&=r\e^{6r\cdot5^n}y^{6\cdot5^{n-1}-3}
((4y+1)^{5^{n-1}(20r-4)+2}+(4y+1)^{5^{n-1}(10r-2)+1})\\
&=r\e^{6r\cdot5^n}y^{6\cdot5^{n-1}-3}
((4y^{5^{n-1}}+1)^{20r-4}(4y+1)^2
+(4y^{5^{n-1}}+1)^{10r-2}(4y+1))\\
&=r\e^{6r\cdot5^n}y^{6\cdot5^{n-1}-3}(2+2y+O(y^2))\\
&=r\e^{6r\cdot5^n}(2y^{6\cdot5^{n-1}-3}
+2y^{6\cdot5^{n-1}-2}+O(y^{a_n}))
\end{align*}
in $\MTZF$.  %This gives Equation...
\end{proof}
%\begin{lem}\label{D0}
%In $\MTZF$, $L_2D_{0,n,r}=r\e^{6r\cdot5^n}(
%y^{27\cdot5^{n-2}}+y^{28\cdot5^{n-2}}+3y^{29\cdot5^{n-2}}
%+O(y^{a_n}))$.
%\end{lem}
%\begin{proof}
%bleh
%\end{proof}
%\newpage
%\begin{lem}\label{DNMINUS2}
%In $\MTZF$, $L_2D_{n-2,n,r}=r\e^{6r\cdot5^n}
%%(2y^{29\cdot5^{n-2}+2}
%%+2y^{29\cdot5^{n-2}+3}
%(2y^{6\cdot5^{n-1}-3}
%+2y^{6\cdot5^{n-1}-2}
%+O(y^{a_n}))$.
%\end{lem}
%\begin{proof}
%bleh
%\end{proof}
%\newpage
\begin{lem}\label{CDGENERAL}
For $1\leq m\leq n-2$,
\begin{equation*}
L_2C_{m,n,r}=r\e^{6r\cdot5^n}
(4y^{6\cdot5^{n-1}-5^{n-m-1}}
+3y^{6\cdot5^{n-1}-3\cdot5^{n-m-2}}
+3y^{6\cdot5^{n-1}-2\cdot5^{n-m-2}}
+O(y^{a_n}))
\end{equation*}
and for $1\leq m\leq n-3$,
\begin{equation*}
L_2D_{m,n,r}=r\e^{6r\cdot5^n}
(2y^{6\cdot5^{n-1}-3\cdot5^{n-m-2}}
+2y^{6\cdot5^{n-1}-2\cdot5^{n-m-2}}
+y^{6\cdot5^{n-1}-5^{n-m-2}}
+O(y^{a_n}))
\end{equation*}
in $\MTZF$.
\end{lem}
\begin{proof}
First, we compute $L_2C_{m,n,r}$
for $1\leq m\leq n-2$.
By Lemma \ref{L2TERM},
\begin{align}\label{CM1}
\begin{split}
L_2C_{m,n,r}&=3rL_2(
\D^{8\cdot5^{n-1}
	+2\cdot5^{n-m-2}+2(r-1)5^n}
E_4^{6\cdot5^{n-1}-6\cdot5^{n-m-2}}
)\\
&=3r\e^{6\cdot5^n}
y^{6\cdot5^{n-1}-6\cdot5^{n-m-2}}
(
(4y+1)^{16\cdot5^{n-1}+4\cdot5^{n-m-2}
+4(r-1)5^n}
\\
&\hskip1.7truein-(4y+1)^{8\cdot5^{n-1}
+2\cdot5^{n-m-2}+2(r-1)5^n}).
\end{split}
\end{align}
In $\mathbb Z/5[y]$,
\begin{align*}
(4y+1)^{16\cdot5^{n-1}+4\cdot5^{n-m-2}
	+4(r-1)5^n}&=
(4y+1)^{(4r-1)5^n+5^{n-1}+
	4\cdot5^{n-m-2}}\\
&=(4y^{5^n}+1)^{4r-1}
(4y^{5^{n-1}}+1)(4y^{5^{n-m-2}}+1)^4\\
%&=(4y^{5^n}+1)^{4r-1}
%(4y^{5^{n-1}}+1)
%(1+y^{5^{n-m-2}}+y^{2\cdot5^{n-m-2}}
%+y^{3\cdot5^{n-m-2}}
%+y^{4\cdot5^{n-m-2}})\\
&=1+y^{5^{n-m-2}}+y^{2\cdot5^{n-m-2}}
+y^{3\cdot5^{n-m-2}}
+y^{4\cdot5^{n-m-2}}
+O(y^{6\cdot5^{n-m-2}})
\end{align*}
%since $5^{n-1}>6\cdot5^{n-m-2}$ for $m\geq1$,
and
\begin{align*}
(4y+1)^{8\cdot5^{n-1}+2\cdot5^{n-m-2}
+2(r-1)5^n}&=
(4y+1)^{(2r-1)5^n+3\cdot5^{n-1}
+2\cdot5^{n-m-2}}\\
&=(4y^{5^n}+1)^{2r-1}
(4y^{5^{n-1}}+1)^3
(4y^{5^{n-m-2}}+1)^2\\
&=1+3y^{5^{n-m-2}}+y^{2\cdot5^{n-m-2}}
+O(y^{6\cdot5^{n-m-2}})
\end{align*}
which, when combined with \eqref{CM1}, yield
\begin{align*}
L_2C_{m,n,r}&=3r\e^{6\cdot5^n}
y^{6\cdot5^{n-1}-6\cdot5^{n-m-2}}
(3y^{5^{n-m-2}}+y^{3\cdot5^{n-m-2}}
+y^{4\cdot5^{n-m-2}}+O(y^{6\cdot5^{n-m-2}}))\\
&=r\e^{6\cdot5^n}(4y^{6\cdot5^{n-1}-5^{n-m-1}}
+3y^{6\cdot5^{n-1}-3\cdot5^{n-m-2}}
+3y^{6\cdot5^{n-1}-2\cdot5^{n-m-2}}
+O(y^{a_n}))
\end{align*}
in $\MTZF$. 

Next, we compute $L_2D_{m,n,r}$
for $1\leq m\leq n-3$.
By Lemma \ref{L2TERM},
\begin{align}\label{DM1}
\begin{split}
L_2D_{m,n,r}&=rL_2(
\D^{8\cdot5^{n-1}
	+\cdot5^{n-m-2}+2(r-1)5^n}
E_4^{6\cdot5^{n-1}-3\cdot5^{n-m-2}}
)\\
&=r\e^{6\cdot5^n}
y^{6\cdot5^{n-1}-3\cdot5^{n-m-2}}
(
(4y+1)^{16\cdot5^{n-1}+2\cdot5^{n-m-2}
+4(r-1)5^n}\\
&\hskip1.7truein
+(4y+1)^{8\cdot5^{n-1}+5^{n-m-2}
+2(r-1)5^n}).
\end{split}
\end{align}
In $\mathbb Z/5[y]$,
\begin{align*}
(4y+1)^{16\cdot5^{n-1}+2\cdot5^{n-m-2}
	+4(r-1)5^n}&=
(4y+1)^{(4r-1)5^n+5^{n-1}
	+2\cdot5^{n-m-2}}\\
&=(4y^{5^n}+1)^{4r-1}(4y^{5^{n-1}}+1)
(4y^{5^{n-m-2}}+1)^2\\
&=1+3y^{5^{n-m-2}}+y^{2\cdot5^{n-m-2}}
+O(y^{3\cdot5^{n-m-2}})
\end{align*}
and	
\begin{align*}
(4y+1)^{8\cdot5^{n-1}+5^{n-m-2}
+2(r-1)5^n}&=
(4y+1)^{(2r-1)5^n+3\cdot5^{n-1}
+5^{n-m-2}}\\
&=(4y^{5^n}+1)^{2r-1}(4y^{5^{n-1}}+1)^3
(4y^{5^{n-m-2}}+1)\\
&=1+4y^{5^{n-m-2}}+O(y^{3\cdot5^{n-m-2}})
\end{align*}
which, when combined with \eqref{DM1}, yield
\begin{align*}
L_2D_{m,n,r}&=r\e^{6\cdot5^n}
y^{6\cdot5^{n-1}-3\cdot5^{n-m-2}}
(2+2y^{5^{n-m-2}}+y^{2\cdot5^{n-m-2}}+O(y^{3\cdot5^{n-m-2}}))\\
&=r\e^{6r\cdot5^n}
(2y^{6\cdot5^{n-1}-3\cdot5^{n-m-2}}
+2y^{6\cdot5^{n-1}-2\cdot5^{n-m-2}}
+y^{6\cdot5^{n-1}-5^{n-m-2}}
+O(y^{a_n}))
\end{align*}
in $\MTZF$.  Note that the assumption $m\leq n-3$ is essential,
because if $m=n-2$,
\[
6\cdot5^{n-1}-5^{n-m-2}=6\cdot5^{n-1}-1=a_n
\]
and so the term $y^{6\cdot5^{n-1}-5^{n-m-2}}
=y^{a_n}$ would
not appear.  This is why the computation
of $L_2D_{n-2,n,r}$ is handled separately in Lemma \ref{CDEXTREME}.
\end{proof}
%\begin{prop}\label{DM}
%For $1\leq m\leq n-3$,
%\[
%L_2D_{m,n,r}=r\e^{6r\cdot5^n}
%(2y^{6\cdot5^{n-1}-3\cdot5^{n-m-2}}
%+2y^{6\cdot5^{n-1}-2\cdot5^{n-m-2}}
%+y^{6\cdot5^{n-1}-5^{n-m-2}}
%+O(y^{a_n}))
%\]
%in $\MTZF$.
%\end{prop}
%\begin{proof}
%%MY PLACE
%bleh,.
%\end{proof}
% DIVISIBILITY
\begin{thm}\label{divisibility}
The modular forms $f_{r\cdot 5^n/j}$
as identified in Theorem \ref{main_5}
satisfy condition (C4) of Behrens' theorem.
%in the case $\ell=2$.
\end{thm}
\begin{proof}
%Condition (C4) says that $L_2f_{r\cdot 5^n/j}$
%must be congruent to a form $g\in M_{24r\cdot 5^n-4j}(2)$
%modulo $5$.  By Proposition \ref{Serre}, 
By Corollary \ref{C1C3}, Proposition \ref{C2ONLY}, 
and Lemma \ref{rigidlem},
it suffices
to show that $L_2f_{r\cdot 5^n/j}$ is divisible by
$E_4^j$ in $M_*(2)_{\mathbb Z/5}$.

Consider first the case $j\leq 5^n$, so that
$f_{r\cdot 5^n/j}=\D^{2r\cdot 5^n}$.
Lemma \ref{L2ONDELTA} implies that
\[
L_2\D^{2r\cdot 5^n}=\e^{6\cdot5^n}(3ry^{5^n}+O(y^{a_n}))
\in\MTZF
\]
showing divisibility by $E_4^{5^n}$
and verifying condition (C4) in this case.

Next, suppose $j>5^n$, and let $u$ be the positive integer
between 1 and $n-1$ such that
\[
5^n+5^{n-1}-5^{n-u}+1\leq j\leq 5^n+5^{n-1}-5^{n-u-1}.
\]   Assume first
that $j>5^n+5^{n-1}-5^{n-u}+2\cdot5^{n-u-1}$,
so that the modular form in question is
\[
f_{r\cdot5^n/j}=\D^{2r\cdot5^n}+
\sum_{m=0}^{u-1}(C_{m,n,r}+D_{m,n,r}).
\]
To show that $L_2f_{r\cdot5^n/j}$
is divisible by $E_4^j$
in $\MTZF$, it suffices to show divisibility
by $E_4^{5^n+5^{n-1}-5^{n-u-1}}$
since $5^n+5^{n-1}-5^{n-u-1}$ is the largest
possible $j$-value in this case.
By Lemmas \ref{L2ONDELTA} and \ref{CDEXTREME}, 
\begin{equation}\label{L2ONSTART}
L_2(\D^{2r\cdot5^n}+C_{0,n,r}+D_{0,n,r})=
\e^{6r\cdot5^n}
(ry^{29\cdot5^{n-2}}+O(y^{a_n}))
\end{equation}
in $\MTZF$.
Lemmas \ref{CDEXTREME} and \ref{CDGENERAL}
imply
\begin{equation}\label{L2ONCDNM2}
L_2(C_{n-2,n,r}+D_{n-2,n,r})=\e^{6r\cdot5^n}
(4ry^{6\cdot5^{n-1}-5}+O(y^{a_n}))\in\MTZF
\end{equation}
and for $1\leq m\leq n-3$,
Lemma \ref{CDGENERAL} implies
\begin{equation}\label{L2ONCD}
L_2(C_{m,n,r}+D_{m,n,r})=\e^{6r\cdot5^n}
(4ry^{6\cdot5^{n-1}-5^{n-m-1}}
+ry^{6\cdot5^{n-1}-5^{n-m-2}}+O(y^{a_n}))
\in\MTZF.
\end{equation}
Therefore, for $1\leq u\leq n-2$,
\begin{align}\label{L2ONCOMBO}
\begin{split}
L_2f_{r\cdot5^n/j}&=\e^{6r\cdot5^n}
\left(
ry^{29\cdot5^{n-2}}+\sum_{m=1}^{u-1}
(4ry^{6\cdot5^{n-1}-5^{n-m-1}}
+ry^{6\cdot5^{n-1}-5^{n-m-2}})+O(y^{a_n})
\right)\\
&=\e^{6r\cdot5^n}(ry^{6\cdot5^{n-1}-5^{n-u-1}}
+O(y^{a_n}))\in\MTZF
\end{split}
\end{align}
by Equations \eqref{L2ONSTART} and \eqref{L2ONCD},
showing divisibility by $E_4^{5^n+5^{n-1}-5^{n-u-1}}$.
For the case $u=n-1$, we obtain
\[
L_2f_{r\cdot5^n/j}=\e^{6r\cdot5^n}
(ry^{6\cdot5^{n-1}-5}
+4ry^{6\cdot5^{n-1}-5}+O(y^{a_n}))
=\e^{6r\cdot5^n}\cdot O(y^{a_n})
\]
from Equations \eqref{L2ONCDNM2} and \eqref{L2ONCOMBO},
showing divisibility
by $E_4^{5^n+5^{n-1}-5^{n-(n-1)-1}}
=E_4^{a_n}$.  Thus, for
$j>5^n+5^{n-1}-5^{n-u}+2\cdot5^{n-u-1}$,
condition (C4) is verified.

Continuing with the case $j>5^n$, assume
now that $j\leq
5^n+5^{n-1}-5^{n-u}+2\cdot5^{n-u-1}$,
so that the modular form in question is
\[
f_{r\cdot5^n/j}=\D^{2r\cdot5^n}+
\sum_{m=0}^{u-2}(C_{m,n,r}+D_{m,n,r})
+C_{u-1,n,r}.
\]
In this case, it suffices to show
divisibility by
$E_4^{5^n+5^{n-1}-5^{n-u}+2\cdot5^{n-u-1}}$
since $5^n+5^{n-1}-5^{n-u}+2\cdot5^{n-u-1}$
is the largest possible value of $j$.
For $u=1$, Lemmas \ref{L2ONDELTA}
and \ref{CDEXTREME} together imply
\begin{align*}
L_2f_{r\cdot5^n/j}&=L_2(\D^{2r\cdot5^n}
+C_{0,n,r})\\
&=\e^{6r\cdot5^n}(3ry^{5^n}+2ry^{5^n}
+4ry^{27\cdot5^{n-2}}
+4ry^{28\cdot5^{n-2}}
+3ry^{29\cdot5^{n-2}}
+O(y^{a_n}))\\
&=\e^{6r\cdot5^n}(
4ry^{27\cdot5^{n-2}}
+4ry^{28\cdot5^{n-2}}
+3ry^{29\cdot5^{n-2}}
+O(y^{a_n}))\in\MTZF
\end{align*}
showing divisibility by
$E_4^{5^n+5^{n-1}-5^{n-1}+2\cdot5^{n-2}}
%=E_4^{5^n+2\cdot5^{n-2}}
=E_4^{27\cdot5^{n-2}}$.
For $2\leq u\leq n-1$, we obtain
\begin{align*}
L_2f_{r\cdot5^n/j}&=\e^{6r\cdot5^n}
\bigg(ry^{29\cdot 5^{n-2}}
+\sum_{m=1}^{u-2}(4ry^{6\cdot5^{n-1}-5^{n-m-1}}
+ry^{6\cdot5^{n-1}-5^{n-m-2}})\\
&\hskip1.2truein
+4ry^{6\cdot5^{n-1}-5^{n-u}}
+3ry^{6\cdot5^{n-1}-3\cdot5^{n-u-1}}
+3ry^{6\cdot5^{n-1}-2\cdot5^{n-u-1}}
+O(y^{a_n})
\bigg)\\
&=\e^{6r\cdot5^n}
(3ry^{6\cdot5^{n-1}-3\cdot5^{n-u-1}}
+3ry^{6\cdot5^{n-1}-2\cdot5^{n-u-1}}
+O(y^{a_n}))\in\MTZF
\end{align*}
from
Equations \eqref{L2ONSTART},
\eqref{L2ONCDNM2}, and 
\eqref{L2ONCD},
showing divisibility by
\[
E_4^{5^n+5^{n-1}-5^{n-u}+2\cdot5^{n-u-1}}=
E_4^{6\cdot5^{n-1}-3\cdot5^{n-u-1}}.
\]
Thus, condition (C4) is verified for
this last remaining case. 
%of
%$j\leq5^n+5^{n-1}-5^{n-u}+2\cdot5^{n-u-1}$.
\end{proof}

Theorem \ref{main_5} follows from
Corollary \ref{C1C3}, Proposition \ref{C2ONLY},
and Theorem \ref{divisibility}.  

\section{Computations at other primes}
\label{atotherprimes}
\subsection{The prime 7}

In the case $p=7$ we are hunting for
$f_{i/j}\in M_{48i}$ associated to each
order 7 generator $\beta_{i/j}$.
The demand on $q$-order in condition (C2) says
\[
\ord_qf_{i/j}>\dfrac{48i-6j}{12}\quad\text{or}\quad
\ord_qf_{i/j}=\dfrac{48i-6j-2}{12}
\]
which is equivalent to $\ord_qf_{i/j}>4i-\dfrac12j$.
In particular, if $j=1$, the inequality
becomes
\begin{equation}\label{qorder7}
\ord_qf_{i}\geq 4i.
\end{equation}
When combined with parts (a) and (c) of
Proposition \ref{setup}, the inequality
\eqref{qorder7} forces
\[
f_i=c\D^{4i}
\]
for some integer $c$ prime to 7.
Taking $c=1$ and recalling that $\beta_i$
is a 7-primary divided beta family element
for all integers $i\geq1$ (see Example \ref{exmp_i})
yields the following theorem.
\begin{thm}\label{thm7}
For any integer $i\geq1$,
the 7-primary divided beta family element $\beta_i$
has corresponding modular form $f_i=\D^{4i}$.
\end{thm}
\begin{rem}
Proposition \ref{rigidity} does not apply
in the case $\ell_0=2$ and $p=7$
since 2 is not a topological generator
of $\mathbb Z_7^{\times}$.  Further computations at
the prime 7 could be done by working with
$\Gamma_0(3)$ modular forms instead,
since $\ell=3$
{\em is} a topological generator of $\mathbb Z_7^{\times}$.
\end{rem}

\subsection{The prime 11}\label{theprime11} We shall compute
$f_{i/j}\in M_{120i}$ at the prime 11 in
the case $1\leq j\leq 11^n$, where $i=r\cdot11^n$
with $(r,11)=1$.  Since 2 is a
topological generator of $\mathbb Z_{11}^{\times}$,
Proposition \ref{rigidity} applies, and so it suffices to compute
with $\Gamma_0(2)$ modular forms.  In $M_*(2)[1/2]$,
\begin{align*}
E_{10}&=-32768\delta\mu^2-4096\delta\mu\e+1024\delta\e^2,\\
V_2E_{10}&=-32\delta\mu^2+128\delta\mu\e+1024\e^2
\end{align*}
and so in $M_*(2)_{\mathbb Z/11}$,
\begin{equation}\label{E10at11}
E_{10}=V_2E_{10}=\delta\mu^2+7\delta\mu\e+\delta\e^2
=\delta(\mu+3\e)(\mu+4\e).
\end{equation}
Equations \eqref{deltaidentity} and \eqref{vdeltaidentity}
imply the identities
\begin{align}
\D&=9\mu\e^2,\label{deltaat11}\\
V_2\D&=\mu^2\e\label{vdeltaat11}
\end{align}
in $M_*(2)_{\mathbb Z/11}$.
\begin{prop}\label{prop11}
If $f\in M_t$, then $L_2f\in M_t(2)_{\mathbb Z/11}$
is expressible as a homogeneous element of
$\mathbb Z/11[\mu,\e]$ of degree $t/4$, and
if $x=\mu/\e$, then
\begin{equation}\label{lemma11equation}
L_2f=\e^{t/4}P(x)
\end{equation}
for an inhomogeneous polynomial $P(x)\in\mathbb Z/11[x]$.
Moreover, if $P(x)$ is divisible by
$(x+1)^j(x+3)^j(x+4)^j$ in $\mathbb Z/11[x]$,
then $L_2f$ is divisible by $E_{10}^j$ in
$M_*(2)_{\mathbb Z/11}$.
\end{prop}
\begin{proof}
This proposition is analogous to Proposition \ref{method},
and the proof of \eqref{lemma11equation} is similar
to the proof of \eqref{inhomog}.

If $P(x)$ is divisible by
$(x+1)^j(x+3)^j(x+4)^j$ in $\mathbb Z/11[x]$,
then the corresponding homogeneous element
of $\mathbb Z/11[\mu,\e]$ is divisible
by $(\mu+\e)^j(\mu+3\e)^j(\mu+4\e)^j$.
This in turn implies divisibility by $E_{10}^j$
since
\[
(\mu+\e)^j(\mu+3\e)^j(\mu+4\e)^j=\delta^{j}E_{10}^j
\]
by Equation \eqref{E10at11} and the identity
$\delta^2=\mu+\e$.
\end{proof}
\begin{thm}\label{thm11}
Given an 11-primary divided beta family element
$\beta_{r\cdot11^n/j}$ with $(r,11)=1$ and $1\leq j\leq11^n$,
the corresponding modular form is
$f_{r\cdot11^n/j}=\D^{10r\cdot11^n}\in M_{120r\cdot11^n}$.
\end{thm}
\begin{proof}
The modular form $\D^{10r\cdot11^n}$
has nonzero Fourier expansion modulo 11, and
its $q$-order is
\[
\ord_q\D^{10r\cdot11^n}=10r\cdot11^n
>\dfrac{120r\cdot11^n-10j}{12}
\]
for $j\geq1$.  Any power of $\D=E_4^3+10E_6^2$
is not divisible by $E_{10}=E_4E_6$
in $(M_*)_{\mathbb Z/11}$.  Thus, conditions (C1), (C2),
and (C3) are satisfied.

Using the notation from Proposition \ref{prop11}, Equations %\eqref{E10at11}, 
\eqref{deltaat11} and
\eqref{vdeltaat11} imply
\begin{align}\label{11calculation}
\begin{split}
L_2\D^{10r\cdot11^n}&=V_2\D^{10r\cdot11^n}
-\D^{10r\cdot11^n}\\
&=(\mu^{20r}\e^{10r}-9^{10r}\mu^{10r}\e^{10r})^{11^n}\\
&=\e^{30r\cdot11^n}x^{10r\cdot11^n}(x^{10r}-1)^{11^n}\\
&=\e^{30r\cdot11^n}x^{10r\cdot11^n}(x^{10}-1)^{11^n}
(x^{10(r-1)}+x^{10(r-2)}+\cdots+x^{10}+1)^{11^n}.
\end{split}
\end{align}
Moreover, in $\mathbb Z/11[x]$,
\[
x^{10}-1=(x+1)(x+2)(x+3)(x+4)(x+5)(x+6)(x+7)(x+8)(x+9)(x+10)
\]
and so Equation \eqref{11calculation} has the form
$L_2\D^{10r\cdot11^n}=\e^{30r\cdot11^n}P(x)$
where $P(x)$ is divisible by
\[
(x+1)^{11^n}(x+3)^{11^n}(x+4)^{11^n}.
\]
By Proposition \ref{prop11}, this implies
$L_2\D^{10r\cdot11^n}$ is divisible by $E_{10}^{11^n}$
(and hence by $E_{10}^j$ for $1\leq j\leq11^n$)
in $M_*(2)_{\mathbb Z/11}$.
This verifies condition (C4) in the case $\ell=2$ by
Proposition \ref{Serre}.
\end{proof}

\subsection{The primes 13 and 677}
We shall establish analogs of Theorem \ref{thm11}
at the primes 13 and 677.  The proofs will follow
the structure of Subsection \ref{theprime11}
but will be made more concise.
Since 2 is a topological generator of
$\mathbb Z_{13}^{\times}$ and $\mathbb Z_{677}^{\times}$,
it suffices in both cases to compute with $\Gamma_0(2)$
modular forms when checking condition (C4).
\begin{thm}\label{thm13}
Given a 13-primary divided beta family element
$\beta_{r\cdot13^n/j}$ with $(r,13)=1$
and $1\leq j\leq13^n$, the corresponding modular
form is $f_{r\cdot13^n/j}=\D^{14r\cdot13^n}\in
M_{168r\cdot13^n}$.
\end{thm}
\begin{proof}
The modular form $\D^{14r\cdot13^n}$
has nonzero Fourier expansion modulo 11,
and its $q$-order is
\[
\ord_q\D^{14r\cdot13^n}=14r\cdot13^n>\dfrac{168r\cdot13^n-12j}{12}
\]
for $j\geq1$.  Any power of $\D=12E_4^3+E_6^2$
is not divisible by $E_{12}=6E_4^3+8E_6^2$
in $(M_*)_{\mathbb Z/13}$.
Thus, conditions (C1), (C2),
and (C3) are satisfied.

In $M_*(2)_{\mathbb Z/13}$,
\begin{align}
\D&=12\mu\e^2,\label{deltaat13}\\
V_2\D&=\mu^2\e,\label{vdeltaat13}
\end{align}
and
\begin{equation}\label{E12at13}
E_{12}=V_2E_{12}=12\mu^3+9\mu^2\e+4\mu\e^2+\e^3
=12(\mu+12\e)(\mu^2+5\mu\e+\e^2).
\end{equation}
Using the notation from Proposition \ref{prop11}, Equations \eqref{deltaat13} and
\eqref{vdeltaat13} imply
\begin{align}\label{13calculation}
\begin{split}
L_2\D^{14r\cdot13^n}&=V_2\D^{14r\cdot13^n}
-\D^{14r\cdot13^n}\\
&=(\mu^{28r}\e^{14r}-\mu^{14r}\e^{28r})^{13^n}\\
&=\e^{42r\cdot13^n}x^{14r\cdot13^n}(x^{14r}-1)^{13^n}\\
&=\e^{42r\cdot13^n}x^{14r\cdot13^n}(x^{14}-1)^{13^n}
(x^{14(r-1)}+x^{14(r-2)}+\cdots+x^{14}+1)^{13n}.
\end{split}
\end{align}
Moreover, in $\mathbb Z/13[x]$,
\[
x^{14}-1=(x+1)(x+12)
(x^2+3x+1)(x^2+5x+1)(x^2+6x+1)(x^2+7x+1)(x^2+8x+1)
(x^2+10x+1)
\]
and so Equation \eqref{13calculation}
has the form $L_2\D^{14r\cdot13^n}
=\e^{42r\cdot13^n}P(x)$ where $P(x)$
is divisible by
\[
(x+12)^{13^n}(x^2+5x+1)^{13^n}.
\]
Using Equation \eqref{E12at13}
and arguing as in the proof of Proposition
\ref{prop11} shows this is equivalent to
$L_2\D^{14r\cdot13^n}$ being divisible by
$E_{12}^{13^n}$ (and hence by $E_{12}^j$
for $1\leq j\leq13^n$)
in $M_*(2)_{\mathbb Z/13}$.
This verifies condition (C4) in the case $\ell=2$ by
Proposition \ref{Serre}.
\end{proof}
\begin{thm}\label{thm677}
Given a 677-primary divided beta family element
$\beta_{r\cdot677^n/j}$ with $(r,677)=1$
and $1\leq j\leq677^n$, the corresponding modular
form is $f_{r\cdot677^n/j}=\D^{38194r\cdot677^n}\in
M_{458328r\cdot677^n}$.
\end{thm}
\begin{proof}
The modular form $\D^{38194r\cdot677^n}$
has nonzero Fourier expansion modulo 677,
and its $q$-order is
\[
\ord_q\D^{38194r\cdot677^n}=38194r\cdot677^n
>\dfrac{458328r\cdot677^n-676j}{12}
\]
for $j\geq1$.  Any power of $\D$ %=534E_4^3+143E_6^2$
is not divisible by
%$E_{676}=$
%\begin{align}
%\begin{split}
%E_{676}&=66\D^{56}E_4+654\D^{55}E_4^4+539\D^{54}E_4^7
%+476\D^{53}E_4^{10}+95\D^{52}E_4^{13}+176\D^{51}E_4^{16}\\
%&+246\D^{50}E_4^{19}+176\D^{49}E_4^{22}+592\D^{48}E_4^{25}
%+407\D^{47}E_4^{28}+141\D^{46}E_4^{31}+131\D^{45}E_4^{34}\\
%&+55\D^{44}E_4^{37}+164\D^{43}E_4^{40}+63\D^{42}E_4^{43}
%+630\D^{41}E_4^{46}+658\D^{40}E_4^{49}+297\D^{39}E_4^{52}\\
%&+645\D^{38}E_4^{55}+223\D^{37}E_4^{58}+458\D^{36}E_4^{61}
%+337\D^{35}E_4^{64}+603\D^{34}E_4^{67}+231\D^{33}E_4^{70}\\
%&+479\D^{32}E_4^{73}+597\D^{31}E_4^{76}+152\D^{30}E_4^{79}
%+325\D^{29}E_4^{82}+320\D^{28}E_4^{85}+57\D^{27}E_4^{88}\\
%&+424\D^{26}E_4^{91}+252\D^{25}E_4^{94}+656\D^{24}E_4^{97}
%+56\D^{23}E_4^{100}+150\D^{22}E_4^{103}+302\D^{21}E_4^{106}\\
%&+287\D^{20}E_4^{109}+576\D^{19}E_4^{112}+421\D^{18}E_4^{115}
%+90\D^{17}E_4^{118}+141\D^{16}E_4^{121}+27\D^{15}E_4^{124}\\
%&+137\D^{14}E_4^{127}+282\D^{13}E_4^{130}+581\D^{12}E_4^{133}
%+300\D^{11}E_4^{136}+271\D^{10}E_4^{139}+174\D^{9}E_4^{142}\\
%&+96\D^{8}E_4^{145}+383\D^{7}E_4^{148}+544\D^{6}E_4^{151}
%+381\D^{5}E_4^{154}+156\D^{4}E_4^{157}+400\D^{3}E_4^{160}\\
%&+514\D^{2}E_4^{163}+60\D E_4^{166}+E_4^{169}.
%\end{split}
%\end{align}
\[
E_{676}=66\D^{56}E_4+654\D^{55}E_4^4+\cdots+60\D E_4^{166}
+E_4^{169}
\]
in $(M_*)_{\mathbb Z/677}$.
Thus, conditions (C1), (C2),
and (C3) are satisfied.

In $M_*(2)_{\mathbb Z/677}$,
\begin{align}
	\D&=64\mu\e^2,\label{deltaat677}\\
	V_2\D&=\mu^2\e,\label{vdeltaat677}
\end{align}
and
\begin{align}\label{E676at677}
\begin{split}
&E_{676}=V_2E_{676}\\
&=676\mu^{169}+127\mu^{168}\e+236\mu^{167}\e^{2}
+375\mu^{166}\e^{3}+522\mu^{165}\e^{4}+222\mu^{164}\e^{5}
+232\mu^{163}\e^{6}+195\mu^{162}\e^{7}\\
&+220\mu^{161}\e^{8}+22\mu^{160}\e^{9}+461\mu^{159}\e^{10}
+582\mu^{158}\e^{11}+541\mu^{157}\e^{12}+283\mu^{156}\e^{13}
+577\mu^{155}\e^{14}\\
&+598\mu^{154}\e^{15}+263\mu^{153}\e^{16}+361\mu^{152}\e^{17}
+577\mu^{151}\e^{18}+540\mu^{150}\e^{19}+90\mu^{149}\e^{20}
+222\mu^{148}\e^{21}\\
&+248\mu^{147}\e^{22}+164\mu^{146}\e^{23}+494\mu^{145}\e^{24}
+361\mu^{144}\e^{25}+107\mu^{143}\e^{26}+404\mu^{142}\e^{27}
+469\mu^{141}\e^{28}\\
&+265\mu^{140}\e^{29}+21\mu^{139}\e^{30}+4\mu^{138}\e^{31}
+317\mu^{137}\e^{32}+369\mu^{136}\e^{33}+189\mu^{135}\e^{34}
+283\mu^{134}\e^{35}\\
&+490\mu^{133}\e^{36}+543\mu^{132}\e^{37}+81\mu^{131}\e^{38}
+372\mu^{130}\e^{39}+302\mu^{129}\e^{40}+401\mu^{128}\e^{41}
+293\mu^{127}\e^{42}\\
&+199\mu^{126}\e^{43}+532\mu^{125}\e^{44}+49\mu^{124}\e^{45}
+431\mu^{123}\e^{46}+127\mu^{122}\e^{47}+208\mu^{121}\e^{48}
+596\mu^{120}\e^{49}\\
&+277\mu^{119}\e^{50}+222\mu^{118}\e^{51}+325\mu^{117}\e^{52}
+97\mu^{116}\e^{53}+599\mu^{115}\e^{54}+576\mu^{114}\e^{55}
+169\mu^{113}\e^{56}\\
&+152\mu^{112}\e^{57}+528\mu^{111}\e^{58}+273\mu^{110}\e^{59}
+380\mu^{109}\e^{60}+353\mu^{108}\e^{61}+428\mu^{107}\e^{62}
+248\mu^{106}\e^{63}\\
&+478\mu^{105}\e^{64}+327\mu^{104}\e^{65}+529\mu^{103}\e^{66}
+262\mu^{102}\e^{67}+426\mu^{101}\e^{68}+94\mu^{100}\e^{69}
+347\mu^{99}\e^{70}\\
&+474\mu^{98}\e^{71}+59\mu^{97}\e^{72}+210\mu^{96}\e^{73}
+240\mu^{95}\e^{74}+653\mu^{94}\e^{75}+228\mu^{93}\e^{76}
+218\mu^{92}\e^{77}\\
&+262\mu^{91}\e^{78}+518\mu^{90}\e^{79}+508\mu^{89}\e^{80}
+284\mu^{88}\e^{81}+97\mu^{87}\e^{82}+606\mu^{86}\e^{83}
+127\mu^{85}\e^{84}\\
&+550\mu^{84}\e^{85}+71\mu^{83}\e^{86}+580\mu^{82}\e^{87}
+393\mu^{81}\e^{88}+169\mu^{80}\e^{89}+159\mu^{79}\e^{90}
+415\mu^{78}\e^{91}\\
&+459\mu^{77}\e^{92}+449\mu^{76}\e^{93}+24\mu^{75}\e^{94}
+437\mu^{74}\e^{95}+467\mu^{73}\e^{96}+618\mu^{72}\e^{97}
+203\mu^{71}\e^{98}\\
&+330\mu^{70}\e^{99}+583\mu^{69}\e^{100}+251\mu^{68}\e^{101}
+415\mu^{67}\e^{102}+148\mu^{66}\e^{103}+350\mu^{65}\e^{104}
+199\mu^{64}\e^{105}\\
&+429\mu^{63}\e^{106}+249\mu^{62}\e^{107}+324\mu^{61}\e^{108}
+297\mu^{60}\e^{109}+404\mu^{59}\e^{110}+149\mu^{58}\e^{111}
+525\mu^{57}\e^{112}\\
&+508\mu^{56}\e^{113}+101\mu^{55}\e^{114}+78\mu^{54}\e^{115}
+580\mu^{53}\e^{116}+352\mu^{52}\e^{117}+455\mu^{51}\e^{118}
+400\mu^{50}\e^{119}\\
&+81\mu^{49}\e^{120}+469\mu^{48}\e^{121}+550\mu^{47}\e^{122}
+246\mu^{46}\e^{123}+628\mu^{45}\e^{124}+145\mu^{44}\e^{125}
+478\mu^{43}\e^{126}\\
&+384\mu^{42}\e^{127}+276\mu^{41}\e^{128}+375\mu^{40}\e^{129}
+305\mu^{39}\e^{130}+596\mu^{38}\e^{131}+134\mu^{37}\e^{132}
+187\mu^{36}\e^{133}\\
&+394\mu^{35}\e^{134}+488\mu^{34}\e^{135}+308\mu^{33}\e^{136}
+360\mu^{32}\e^{137}+673\mu^{31}\e^{138}+656\mu^{30}\e^{139}
+412\mu^{29}\e^{140}\\
&+208\mu^{28}\e^{141}+273\mu^{27}\e^{142}+570\mu^{26}\e^{143}
+316\mu^{25}\e^{144}+183\mu^{24}\e^{145}+513\mu^{23}\e^{146}
+429\mu^{22}\e^{147}\\
&+455\mu^{21}\e^{148}+587\mu^{20}\e^{149}+137\mu^{19}\e^{150}
+100\mu^{18}\e^{151}+316\mu^{17}\e^{152}+414\mu^{16}\e^{153}
+79\mu^{15}\e^{154}\\
&+100\mu^{14}\e^{155}+394\mu^{13}\e^{156}+136\mu^{12}\e^{157}
+95\mu^{11}\e^{158}+216\mu^{10}\e^{159}+655\mu^{9}\e^{160}
+457\mu^{8}\e^{161}\\
&+482\mu^{7}\e^{162}+445\mu^{6}\e^{163}+455\mu^{5}\e^{164}
+155\mu^{4}\e^{165}+302\mu^{3}\e^{166}+441\mu^{2}\e^{167}
+550\mu\e^{168}+\e^{169}.  
\end{split}
\end{align}
Using the notation from Proposition \ref{prop11}, Equations \eqref{deltaat677} and
\eqref{vdeltaat677} %, and \eqref{E676at677}
imply
\begin{align}\label{677calculation}
\begin{split}
&L_2\D^{38194r\cdot677^n}\\
&=V_2\D^{38194r\cdot677^n}
-\D^{38194r\cdot677^n}\\
&=(\mu^{76388r}\e^{38194r}-\mu^{38194r}\e^{76388r})^{677^n}\\
&=\e^{114582r\cdot677^n}x^{38194r\cdot677^n}(x^{38194r}-1)^{677^n}\\
&=\e^{114582r\cdot677^n}x^{38194r\cdot677^n}(x^{38194}-1)^{677^n}
(x^{38194(r-1)}+x^{38194(r-2)}+\cdots+x^{38194}+1)^{677n}.
\end{split}
\end{align}
Moreover, in $\mathbb Z/677[x]$, one can check
that the inhomogeneous polynomial
\[
676x^{169}+127x^{168}+\cdots+550x+1
\]
corresponding to the homogeneous polynomial
in $\mu$ and $\e$ in Equation
\eqref{E676at677} divides $x^{38194}-1$.
%and so Equation \eqref{677calculation}
%has the form $L_2\D^{38194r\cdot677^n}
%=\e^{()r\cdot677^n}P(x)$ where $P(x)$
%is divisible by
%\[
%(x+12)^{13^n}(x^2+5x+1)^{13^n}.
%\]
%Using Equation \eqref{E676at677}
%and arguing as in the proof of Proposition
%\ref{prop11} shows
As in the cases $p=11$ and $p=13$, Equation
\eqref{677calculation} implies this is equivalent to
$L_2\D^{38194r\cdot677^n}$ being divisible by
$E_{676}^{677^n}$ (and hence by $E_{676}^j$
for $1\leq j\leq677^n$)
in $M_*(2)_{\mathbb Z/677}$.
This verifies condition (C4) in the case $\ell=2$ by
Proposition \ref{Serre}.
\end{proof}

\subsection{A conjecture}
Theorems \ref{main_5}, \ref{thm7},
\ref{thm11},
\ref{thm13}, and
\ref{thm677} all give evidence for
the following conjecture.
\begin{conj}\label{conj}
Let $p\geq5$ be a prime.  If $1\leq i\in\mathbb Z$
and $i=rp^n$ with $(r,p)=1$, then
\[
f_{i/j}=\D^{i(p^2-1)/12}
\]
for $1\leq j\leq p^n$,
and $f_{i/j}\neq\D^{i(p^2-1)}$
for all other values of $j$
allowed by Lemma \ref{betas}.
\end{conj}
In theory, one possible approach to
Conjecture \ref{conj} 
would be to obtain a closed
formula for the Eisenstein series $E_{p-1}
\in M_{p-1}(2)_{\mathbb Z/p}$
as a polynomial in $\mu$ and $\e$
(and $\delta$ if $p\equiv 3\mod 4$). 
Specific instances of this include Equations
\eqref{E10at11}, \eqref{E12at13}, and \eqref{E676at677},
each obtainable by employing the recursive formula
\[
(t-3)(2t-1)(2t+1)\dfrac{B_{2t}}{(2t)!}
E_{2t}=-3\sum_{a+b=t}
\dfrac{(2a-1)(2b-1)B_{2a}B_{2b}}{(2a)!(2b)!}
E_{2a}
E_{2b}
\]
with $t=(p-1)/2$ and then reducing modulo $p$ for $p=11$, 13, and 677, 
respectively.  
Another approach 
would be to find a more conceptual reason why
$E_{p-1}^{p^n}$ should always divide
$L_{\ell}\D^{i(p^2-1)/12}$
in $M_*(\ell)_{\mathbb Z/p}$.  

The next logical step beyond Conjecture \ref{conj}
is to generalize Theorem \ref{main_5}
by computing all modular forms $f_{i/j}$
at arbitrary primes $p\geq7$.  We are hopeful
that the ``correction terms'' used in this paper
at the prime 5 somehow have natural $p$-primary analogs.
In particular, one can first try to generalize 
Equation \eqref{FTFTN} by producing modular
forms $f_{p^2/a_2}$ satisfying conditions (C1) through
(C4) for $i=p^2$ and $j=a_2$.  

\subsection*{Acknowledgments}  The author would like to
thank Amanda Folsom, Tom Shemanske, and Mark Behrens
for helpful correspondence.

\bibliographystyle{plain}
\bibliography{math}

\end{document}